\documentclass[12pt,reqno]{amsart}
\usepackage{amsfonts,color,amsmath,amssymb,fancyhdr}
\usepackage{amsthm}
\usepackage{txfonts}
\usepackage{float}
\usepackage{hyperref}
\usepackage{cite}

\usepackage{booktabs}

\hypersetup{hidelinks,pdfstartview=Fit,breaklinks=true}
\def\Box{\vcenter{\vbox{\hrule\hbox{\vrule
     \vbox to 8.8pt{\hbox to 10pt{}\vfill}\vrule}\hrule}}}

\newtheorem{thm}{Theorem}[section]
\newtheorem{lemma}[thm]{Lemma}
\newtheorem{corollary}[thm]{Corollary}

\newtheorem{definition}[thm]{Definition}

\numberwithin{equation}{section}
\newtheorem{remark}[thm]{Remark}

\definecolor{Purple}{rgb}{0.5,0,0.5}

   \def\PG{{\rm PG}}  
  \def\Tr{{\rm Tr}}

\begin{document}
\pagestyle{plain}

\title{ On intriguing sets in five classes of strongly regular graphs}

\begin{center}
\author{Xiufang Sun, Jianbing Lu$^*$
}\end{center}
\address{School of Mathematical Sciences, Nankai University, Tianjin 300071, China}
\email{xfsun@mail.nankai.edu.cn}

\address{School of Mathematical Sciences, Zhejiang University, Hangzhou 310027,  China}
\email{jianbinglu@zju.edu.cn}

\begin{abstract}  In this paper, we construct intriguing sets in five classes of strongly regular graphs defined on nonisotropic points of finite classical polar spaces, and determine their intersection numbers.
\end{abstract}

\maketitle
\section{Introduction}\label{introduction}
Suppose that $\Gamma=(X,E)$ is a strongly regular graph and $Y\subseteq X$ is a set of vertices of $\Gamma$. The set $Y$ is called an \emph{intriguing set} of $\Gamma$ if there exist two constants $h_{1}$ and $h_{2}$ such that every vertex $y\in Y$ is adjacent to $h_{1}$ vertices of $Y$ and every vertex $z\notin Y$ is adjacent to $h_{2}$ vertices of $Y$. Furthermore, $Y$ is said to be \emph{nontrivial} if $Y$ is a nonempty proper subset of $X$.

In \cite{ref1}, De Bruyn and Suzuki developed the theory of intriguing sets of regular graphs and described their connections with other notions and results from algebraic graph theory. Intriguing sets of strongly regular graphs which arise as collinearity graphs of point-line geometries have been extensively studied, such as generalized quadrangles \cite{ref2}, finite polar spaces \cite{ref3,ref4,ref5,ref6} and partial quadrangles \cite{ref7}. In 1987, Payne \cite{ref8} introduced the concept of \emph{tight sets} of a generalized quadrangle. The concept of \emph{$m$-ovoids} of a generalized quadrangle was introduced by Thas in \cite{ref9}. It was shown that an intriguing set of a generalized quadrangle is an $m$-ovoid or a tight set \cite{ref2}. Bamberg et al studied intriguing sets in finite polar spaces of higher rank \cite{ref3}. The notions of tight sets and  $m$-ovoids of finite polar spaces can be unified under the umbrella of intriguing sets of finite polar spaces. More intriguing sets of polar spaces have been constructed, see \cite{ref17,ref18,ref19,ref20,ref21}. Bamberg et al proved that the $m$-ovoids of $H(2r,q^{2})$, $Q^{-}(2r+1,q)$, $W(2r-1,q)$ and the tight sets of $H(2r-1,q^{2})$, $Q^{+}(2r-1,q)$, $W(2r-1,q)$ are \emph{two-intersection sets} \cite{ref3}. Calderbank and Kantor surveyed the relations between two-weight codes, two-intersection sets, and strongly regular graphs \cite{ref10}.

In this paper, we study intriguing sets in  five classes of strongly regular graphs related to finite classical polar spaces. We provide three methods to construct the intriguing sets in these strongly regular graphs by using the totally singular subspaces and properly chosen subgroups of the associated classical groups.

This paper is organized as follows. In Section \ref{s1}, we present some concepts and properties about intriguing sets in strongly regular graphs and preliminaries about quadratic forms and hermitian forms. In Section \ref{s2}, we give an explicit description of the five classes of strongly regular graphs. In Section \ref{s3}, we present three constructions of intriguing sets in five classes of strongly regular graphs.

\section{Preliminaries}\label{s1}

\subsection{}
\textbf{Strongly Regular Graphs}\\

A regular graph $\Gamma=(X,E)$ with $v$ vertices is called strongly regular with parameters $(v,k,\lambda,\mu)$ if each vertex has degree $k$ and furthermore that any two adjacent vertices are both adjacent to $\lambda$ vertices while any two non-adjacent vertices are both adjacent to $\mu$ vertices. We use the abbreviation $srg(v,k,\lambda,\mu)$ for a strongly regular graph with parameters $(v,k,\lambda,\mu)$. If $A$ is the adjacency matrix of the strongly regular graph $\Gamma$, then $A$ has three eigenvalues $k,e^+>0,e^-<0$, where $e^+e^-=\mu-k$ and $e^{+}+e^-=\lambda-\mu$. Please refer to \cite{ref12} for more details.
\subsection{}
\textbf{Intriguing Sets in Strongly Regular Graphs}\\

Let $Y\subseteq X$ be a set of vertices of $\Gamma$. The set $Y$ is said to be intriguing if there exist two constants $h_{1}$ and $h_{2}$ such that every vertex $y\in Y$ is adjacent to $h_{1}$ vertices of $Y$ and every vertex $z\notin Y$ is adjacent to $h_{2}$ vertices of $Y$. We call $h_{1}$ and $h_{2}$ the \emph{intersection numbers} of $Y$. Clearly, $\varnothing$ and $X$ are examples of intriguing sets.
\begin{lemma}\label{2.1}
\cite[Lemma 2]{ref3}Suppose $\Gamma=(X,E)$ is a $srg(v,k,\lambda,\mu)$ and $A$ is the adjacency matrix of $\Gamma$. Let $Y$ be an intriguing set of $\Gamma$ with intersection numbers $h_1$ and $h_2$. Then $(h_1-h_2-k)j_Y+h_2j$ is an eigenvector of $A$ with eigenvalue $h_1-h_2$, where $j_Y$ and $j$ are characteristic column vectors of $Y$ and $X$, respectively.
\end{lemma}
By Lemma \ref{2.1}, we define $Y$ to be a \emph{positive} or \emph{negative} intriguing set according to whether $h_{1}-h_{2}$ is equal to $e^+$ or $e^-$.
\begin{definition}
Suppose that $\Gamma=(X,E)$ is a $srg(v,k,\lambda,\mu)$ and $k$, $e^+>0$, $e^-<0$ are eigenvalues of $\Gamma$. For a vertex $P\in X$, let $N(P)$ be the set of vertices in $X$ adjacent to $P$. Let $Y$ be a subset of $X$.

If there exists a constant a such that
\[
|N(P)\cap Y |=
\left \{
\begin{array}{ll}
	a+e^+         &\text{if}\;P\in Y,\\
	a             &\text{if}\;P\in X \setminus Y,\\
\end{array}
\right.
\]
then $Y$ is called a positive intriguing set of $\Gamma$.

If there exists a constant b such that
\[
|N(P)\cap Y |=
\left \{
\begin{array}{ll}
	b+e^-         &\text{if}\;P\in Y,\\
	b             &\text{if}\;P\in X \setminus Y,\\
\end{array}
\right.
\]
then $Y$ is called a negative intriguing set of $\Gamma$.
\end{definition}

One can obtain new intriguing sets by taking unions of disjoint intriguing sets of the same type. Moreover, the complement of an intriguing set is also intriguing, and of the same type. These observations will be important in the study of intriguing sets of strongly regular graphs.
\begin{lemma}\label{2}
\cite[Lemma 2.3]{ref7} Suppose that $\Gamma$ is a strongly regular graph. Let $Y_{1}$ and $Y_{2}$ be two intriguing sets of the same type. Then:

(a) If $Y_{1} \subset Y_{2}$, then $Y_{2} \backslash Y_{1}$ is an intriguing set of the same type as $Y_{2}$ and $Y_{1}$;

(b) If $Y_{1}$ and $Y_{2}$ are disjoint, then $Y_{1}\cup Y_{2}$ is an intriguing set of the same type as $Y_{1}$ and $Y_{2}$;

(c) The complement $Y_{1}^{\prime}$ of $Y_{1}$ in $\Gamma$ is an intriguing set of the same type as $Y_{1}$.
\end{lemma}
\subsection{}
\textbf{Quadratic Forms and Hermitian Forms}\\

Let $V$ be a vector space over a finite field $\mathbb{F}$. A quadratic form $Q$ on $V$ is a map from $V$ to $\mathbb{F}$ such that (i) $Q(\lambda u)=\lambda^{2}Q(u)$ for all $\lambda\in\mathbb{F}$ and $u\in V$; (ii) $B(u,v)=Q(u+v)-Q(u)-Q(v)$ is a bilinear form on $V$. We say that $B$ is the polar form of $Q$. There are three types of nondegenerate quadratic forms: hyperbolic, parabolic, elliptic. A vector $u\in V$ is singular if $Q(u)=0$ and a subspace $U$ of $V$ is totally singular if $Q(u)=0$ for any $u\in U$.

A hermitian form $H$ on $V$ is a map from $V\times V$ to $\mathbb{F}$ with the property (i) $H(au+bv,w)=aH(u,w)+bH(v,w)$ for all $u,v,w\in V$, and $a,b\in\mathbb{F}$; (ii) $H(u,v)=H(v,u)^{\sigma}$ for all $u,v\in V$, where $\sigma$ is an automorphism of $\mathbb{F}$, $\sigma^{2}=id,\sigma\neq id$.

An isometry with respect to a quadratic form $Q$ (resp. hermitian form $H$) is a linear map $\theta$ from $V$ to $V$ with the property that $Q(u^{\theta})=Q(u)$ for all $u\in V$ (resp. $H(u^{\theta},v^{\theta})=H(u,v)$ for all $u,v\in V$), and the set of all isometries of $V$ forms an orthogonal group (resp. unitary group). Please refer to \cite{ref11} for more details.

The following lemma will be used to prove our main results.
\begin{lemma}\label{witt}
\cite[Theorem 3.15]{ref15} Let $Q: V \rightarrow \mathbb{F}$ be a quadratic form, and let $\theta:U\rightarrow W$ be an isometry between two subspaces of $V$. Then $\theta$ extends to an isometry of $V$ to itself, i.e. there is an isometry $\hat{\theta}: V \rightarrow V$ such that $\hat{\theta}|_{U}=\theta$ if and only if $(U\cap V^{\perp})^{\theta}=W\cap V^{\perp}$.
\end{lemma}
This lemma is commonly known as ``Witt's extension lemma". In particular, if $V^{\perp}=0$, then any isometry between subspaces of $V$ extends to an isometry of $V$.

\subsection{}
\textbf{Notation}\\

In this paper, we denote the set of all nonzero squares of $\mathbb{F}_{q}$ by $\square_{q}$, the set of nonsquares by $\blacksquare_{q}$ for $q$ odd. If $x$ is a vector of $V$, let $\langle x\rangle$ be the corresponding point in $\PG(V)$. For a group $G$ acting on a non-empty set $X$, let $x^{g}$ denote the image of $g$ acting on x, where $x\in X$ and $g\in G$. For a point $\langle x\rangle\in \PG(V)$, let $\langle x\rangle^G$ denote the orbit containing $\langle x\rangle$ under the action of $G$. Let $\Tr_{q^{m}/q^{n}}$ be the trace function from $\mathbb{F}_{q^{m}}$ to $\mathbb{F}_{q^{n}}$.

\section{Five classes of strongly regular graphs}\label{s2}
In \cite{ref22}, five classes of strongly regular graphs ($NO_{2r+1}^{\epsilon\perp}(q)$ with $q=3, 5$; $NO_{2r}^{\epsilon}(3)$; $NO_{2r}^{\epsilon}(2)$; $NO_{2r+1}^{\epsilon}(q)$; $NU_{2r}(q)$) and their parameters are given. We now give an explicit description of the five classes of strongly regular graphs.
\subsection{}\label{graph1}\
\textbf{$NO_{2r+1}^{\epsilon\perp}(q)$ with $q=3, 5$}\\

Let $V$ be a vector space of dimension $2r+1$ over $\mathbb{F}_{q}$, where $r\geq1$, provided with a non-degenerate quadratic form $Q$. The set of non-singular points is split into two parts, depending on the type $\epsilon$ $(=\pm1)$ of the hyperplane $x^\perp$ ($\epsilon=+1$ or $-1$ corresponds to the value of $Q$ in $\square_{q}$ or $\blacksquare_{q}$). Let $\Gamma_1^\epsilon$ be the graph on one part, where two points are adjacent when they are orthogonal. Then $\Gamma_1^\epsilon$ is a $srg(v,k,\lambda,\mu)$ with
\[v=\frac{q^r(q^r+\epsilon)}{2},k=\frac{q^{r-1}(q^r-\epsilon)}{2}, \lambda= \left \{
\begin{array}{ll}
	\frac{3^{r-1}(3^{r-1}-\epsilon)}{2}           &\text{if}\;q=3,\\
	\frac{5^{r-1}(5^{r-1}+\epsilon)}{2}           &\text{if}\;q=5,\\
\end{array}
\right. \mu=\frac{q^{r-1}(q^{r-1}-\epsilon)}{2}.\]
We denote this graph by $NO_{2r+1}^{\epsilon\perp}(q)$.

\subsection{}\label{graph2}\
\textbf{$NO_{2r}^{\epsilon}(3)$}\\

Let $V$ be a vector space of dimension $2r$ over $\mathbb{F}_{3}$, where $r\geq2$, provided with a non-degenerate quadratic form $Q$ of type $\epsilon$, $\epsilon=\pm1$. The set of non-singular points is split into two parts of equal size by considering the value of $Q$. Let $\Gamma_2^\epsilon$ be the graph on one part, where two points are adjacent when they are orthogonal. Then $\Gamma_2^\epsilon$ is a $srg(v,k,\lambda,\mu)$  with \[v=\frac{3^{r-1}(3^r-\epsilon)}{2},k=\frac{3^{r-1}(3^{r-1}-\epsilon)}{2},\lambda=\frac{3^{r-2}(3^{r-1}+\epsilon)}{2},\mu=\frac{3^{r-1}(3^{r-2}-\epsilon)}{2}.\]
We denote this graph by $NO_{2r}^{\epsilon}(3)$.

\subsection{}\label{graph3}\
\textbf{$NO_{2r}^{\epsilon}(2)$}\\

Let $V$ be a vector space of dimension $2r$ over $\mathbb{F}_{2}$, where $r\geq2$, provided with a non-degenerate quadratic form $Q$ of type $\epsilon$, $\epsilon=\pm1$. Let $\Gamma_3^\epsilon$ be the graph on non-singular points, adjacent when they are orthogonal. Then $\Gamma_3^\epsilon$ is a $srg(v,k,\lambda,\mu)$ with \[v=2^{2r-1}-\epsilon2^{r-1},\,k=2^{2r-2}-1,\,\lambda=2^{2r-3}-2,\,\mu=2^{2r-3}+\epsilon2^{r-2}.\]
We denote this graph by $NO_{2r}^{\epsilon}(2)$.

\subsection{}\label{graph4}\
\textbf{$NO_{2r+1}^{\epsilon}(q)$ with $q$ odd}\\

Let $V$ be a vector space of dimension $2r+1$ over $\mathbb{F}_{q}$, where $r\geq1$, provided with a non-degenerate quadratic form $Q$. The set of non-singular points is split into two parts, depending on the type $\epsilon$ $(=\pm1)$ of the hyperplane $x^\perp$ ($\epsilon=+1$ or $-1$ corresponds to the value of $Q$ in $\square_{q}$ or $\blacksquare_{q}$). Let $\Gamma_4^\epsilon$ be the graph on one part, where two points are adjacent when the line joining them is a tangent. Then $\Gamma_4^\epsilon$ is a $srg(v,k,\lambda,\mu)$ (also see \cite{ref23}) with parameters
\[v=\frac{q^{r}(q^{r}+\epsilon)}{2}, k=(q^{r-1}+\epsilon)(q^{r}-\epsilon), \lambda=2(q^{2r-2}-1)+\epsilon q^{r-1}(q-1), \mu=2q^{r-1}(q^{r-1}+\epsilon).\]
We denote this graph by $NO_{2r+1}^{\epsilon}(q)$. Note that the line $\langle x,y\rangle$ is a tangent if and only if there exists unique $\lambda\in\mathbb{F}_{q}$ such that $Q(\lambda x+y)=0$, i.e. $B(x,y)^{2}=4Q(x)Q(y)$.

\begin{remark}
In the special case $q=3$, the complement of ${NO_{2r+1}^{\epsilon}(3)}$ is $NO_{2r+1}^{\epsilon\perp}(3)$. Moreover, a set $Y$ is a positive (resp. negative) intriguing set of a strongly regular graph $\Gamma$ if and only if $Y$ is a positive (resp. negative) intriguing set of the complement $\overline{\Gamma}$.
\end{remark}

\subsection{}\label{graph5}\
\textbf{$NU_{2r}(q)$ with $q$ even and $r$ odd}\\

Let $V$ be a vector space of dimension $2r$ over $\mathbb{F}_{q^2}$, where $r\geq2$, provided with a non-degenerate hermitian form $H$. Let $\Gamma_5$ be the graph on nonisotropic points, adjacent when joined by a tangent. Then $\Gamma_5$ is a $srg(v,k,\lambda,\mu)$ with
$$
\begin{array}{ll}
v=q^{2r-1}(q^{2r}-1)/(q+1), & k=(q^{2r-1}+1)(q^{2r-2}-1), \\
\lambda=q^{4r-5}(q+1)-q^{2r-2}(q-1)-2, & \mu=q^{2r-3}(q+1)(q^{2r-2}-1).
\end{array}
$$
We denote this graph by $NU_{2r}(q)$.
In this paper, for $\Gamma_{5}$, we only consider $p=2,q=p^{k},V=\mathbb{F}_{q^{2r}}\times\mathbb{F}_{q^{2r}}$, where $r$ is odd. For any $(u_1,v_1)$ and $(u_2,v_2)$, define
\[
H((u_1,v_1),(u_2,v_2))=\Tr_{q^{2r}/q^2}(u_1v_2^{q^r}+v_1u_2^{q^r}).
\]
Since $(\Tr_{q^{2r}/q^2}(z))^q=\Tr_{q^{2r}/q^2}(z^{q^r})$ for any $z\in\mathbb{F}_{q^{2r}}$, one can prove that $H$ is a hermitian form.
Let $h((u,v))=H((u,v),(u,v))$. Then $h((u,v))=\Tr_{q^r/q}(uv^{q^r}+vu^{q^r})=\Tr_{q^{2r}/q}(uv^{q^r})$.
\begin{lemma}\label{8}
In $\Gamma_{5}$, two distinct vertices $\langle x\rangle,\langle y\rangle$ are adjacent if and only if $H(x,y)^{q+1}=h(x)h(y)$.
\end{lemma}
\begin{proof}
The claim follows from \cite[Lemma 2.3]{ref24}.
\end{proof}
\section{Construction}\label{s3}
\subsection{Construction I}\

In this section, we construct intriguing sets in $\Gamma_1^\epsilon$, $\Gamma_2^\epsilon$, $\Gamma_3^\epsilon$ and $\Gamma_4^\epsilon$ defined in Section \ref{s2} by using totally singular subspaces.

\begin{thm}\label{th41}
Let $X$ be the vertex set of a strongly regular graph in the table below, and let $W_t$ be a totally singular subspace of dimension $t$, where $1\leq t \leq r$ in $\Gamma_1^\epsilon$ and $\Gamma_4^\epsilon$, $1\leq t \leq r-1$ in $\Gamma_2^\epsilon$ and $\Gamma_3^\epsilon$. Then $W_t^\perp\cap X$ is an intriguing set (see Table \ref{t1}).\\
\begin{table}[htbp]
\centering
\caption{Intriguing sets by totally singular subspaces}\label{t1}
\begin{tabular}{ccccr}

\specialrule{0em}{1.5pt}{1.5pt}
\toprule
$srg$ & $h_{1}$ & $h_{2}$ & \rm{Type} \\

\midrule
$\Gamma_1^\epsilon$ & $ \frac{q^{r-1}(q^{r-t}-\epsilon)}{2}$ & $\frac{q^{r-1}(q^{r-t}+\epsilon)}{2}$ & $\epsilon=1$, \rm{negative}; $\epsilon=-1$, \rm{positive} \\
$\Gamma_2^\epsilon$ & $\frac{3^{r-1}(3^{r-t-1}-\epsilon)}{2}$ & $\frac{3^{r-2}(3^{r-t}-\epsilon)}{2}$ & $\epsilon=1$, \rm{negative}; $\epsilon=-1$, \rm{positive} \\
$\Gamma_3^\epsilon$ & $2^{2r-t-2}-1$ & $2^{2r-t-2}-\epsilon2^{r-2}$ & $\epsilon=1$, \rm{positive}; $\epsilon=-1$, \rm{negative}\\
$\Gamma_4^\epsilon$ & $q^{r-1}(q^{r-t}+\epsilon q-\epsilon)-1$ & $q^{r-1}(q^{r-t}+\epsilon)$ & $\epsilon=1$, \rm{positive}; $\epsilon=-1$, \rm{negative} \\
\bottomrule
\specialrule{0em}{1.5pt}{1.5pt}

\end{tabular}
\end{table}
\end{thm}

\begin{proof}
Take $NO_{2r+1}^{+\perp}(5)$ and $t=1$ as an example. Set the parabolic quadratic form as follows:
\[
Q(x)=x_1x_{2r+1}+x_2x_{2r}+\cdots+x_rx_{r+2}+x_{r+1}^2.
\]
Then the polar form $B$ of $Q$ is given by
\[
B(x,y)=x_1y_{2r+1}+x_{2r+1}y_1+\cdots+x_ry_{r+2}+x_{r+2}y_r+2x_{r+1}y_{r+1}.
\]
By Lemma \ref{witt}, we may suppose $W_1=\langle e_1\rangle$, where $e_{i}$ is a vector whose $i$-th coordinate is $1$ and other positions are $0$. Then
\[
W_1^\perp\cap X=\langle e_1\rangle^\perp\cap X=\{\langle\sum\limits_{i=1}^{2r}a_ie_i\rangle: a_{r+1}^2+\sum\limits_{i=2}^r a_ia_{2r+2-i}\in\square_5, a_i\in\mathbb{F}_{5}\}.
\]
For any $P\in W_1^\perp\cap X$, we can identify $P$ with a vertex in $NO_{2r-1}^{+\perp}(5)$. Hence
\[|N(P)\cap W_1^\perp\cap X|=5\cdot\frac{5^{r-2}(5^{r-1}-1)}{2}=\frac{5^{r-1}(5^{r-1}-1)}{2},\]
where $\frac{5^{r-2}(5^{r-1}-1)}{2}$ is the number of neighbors of each vertex in $NO_{2r-1}^{+\perp}(5)$. For any $P\in X\setminus (W_1^\perp\cap X)$, we suppose $P=\langle \sum\limits_{i=1}^{2r+1} b_i e_i \rangle$ with $Q(\sum\limits_{i=1}^{2r+1} b_i e_i)\in\square_5$ and $b_{2r+1}\neq0$. Then
\[
N(P)\cap W_1^\perp\cap X=\{\langle\sum\limits_{i=1}^{2r}a_ie_i\rangle: a_{r+1}^2+\sum\limits_{i=2}^r a_ia_{2r+2-i}\in\square_5,B(\sum\limits_{i=1}^{2r}a_ie_i,\sum\limits_{i=1}^{2r+1} b_i e_i)=0, a_i\in\mathbb{F}_{5}\}.
\]
If $a_2,...,a_{2r}$ are chosen such that $a_{r+1}^2+\sum\limits_{i=2}^r a_ia_{2r+2-i}\in\square_5$, then there exists unique $a_1$ such that $B(\sum\limits_{i=1}^{2r}a_ie_i,\sum\limits_{i=1}^{2r+1} b_i e_i)=0$. Hence
\[|N(P)\cap W_1^\perp\cap X|=\frac{5^{r-1}(5^{r-1}+1)}{2},\]
where $\frac{5^{r-1}(5^{r-1}+1)}{2}$ is the number of vertices of $NO_{2r-1}^{+\perp}(5)$.
In $\Gamma_4^\epsilon$, for $P\in X\setminus (W_1^\perp\cap X)$, we note that if $a_2,...,a_{2r}$ are chosen such that $a_{r+1}^2+\sum\limits_{i=2}^r a_ia_{2r+2-i}\in\square_5$, then there exist two solutions of $a_1$ such that $B(\sum\limits_{i=1}^{2r}a_ie_i,\sum\limits_{i=1}^{2r+1} b_i e_i)^2=4Q(\sum\limits_{i=1}^{2r}a_ie_i)Q(\sum\limits_{i=1}^{2r+1} b_i e_i)$.

The proofs for other strongly regular graphs follow the same idea as $NO_{2r+1}^{+\perp}(5)$, here we omit the details.
\end{proof}

By Lemma \ref{2}, we obtain new intriguing sets.

\begin{corollary}\label{c1}
Using the notations of Theorem \ref{th41}, $X\setminus (W_t^\perp\cap X)$ is an intriguing set of the same type as $W_t^\perp\cap X$ (see Table \ref{t2}).\\
\begin{table}[htbp]
\centering
\caption{Intriguing sets by complements}\label{t2}
\begin{tabular}{ccccr}

\specialrule{0em}{1.5pt}{1.5pt}
\toprule
$srg$ & $h_{1}$ & $h_{2}$ \\

\midrule
$\Gamma_1^\epsilon$ & $\frac{q^{r-1}(q^r-q^{r-t}-2\epsilon)}{2}$ & $\frac{q^{2r-t-1}(q^t-1)}{2}$ \\
$\Gamma_2^\epsilon$ & $\frac{3^{r-2}(3^r-3^{r-t}-2\epsilon)}{2}$ & $\frac{3^{2r-t-2}(3^t-1)}{2}$ \\
$\Gamma_3^\epsilon$ & $2^{2r-t-2}(2^t-1)+\epsilon2^{r-2}-1$ & $2^{2r-t-2}(2^t-1)$ \\
$\Gamma_4^\epsilon$ & $q^{2r-t-1}(q^t-1)+\epsilon q^{r-1}(q-2)-1$ & $q^{2r-t-1}(q^t-1)$ \\
\bottomrule
\specialrule{0em}{1.5pt}{1.5pt}

\end{tabular}
\end{table}
\end{corollary}

\begin{corollary}\label{c2}
Let $X$ be the vertex set of a strongly regular graph in the table below, and let $W_1\subset W_2\subset\cdots\subset W_r$ be totally singular subspaces of dimension $1,2,\cdots,r$, respectively. Then for any $t\in \{1,2,\cdots,r-1\}$, $(W_t^\perp\cap X)\setminus (W_{t+1}^\perp\cap X)$ is an intriguing set of the same type as $W_t^\perp\cap X$ (see Table \ref{t3}).\\
\begin{table}[htbp]
\centering
\caption{Intriguing sets by set difference}\label{t3}
\begin{tabular}{ccccr}

\specialrule{0em}{1.5pt}{1.5pt}
\toprule
$srg$ & $h_{1}$ & $h_{2}$ \\

\midrule
$\Gamma_1^\epsilon$ & $\frac{q^{r-1}(q^{r-t}-q^{r-t-1}-2\epsilon)}{2}$ & $2\cdot q^{2r-t-2}$ \\
$\Gamma_2^\epsilon$ & $\frac{3^{r-2}(3^{r-t}-3^{r-t-1}-2\epsilon)}{2}$ & $3^{2r-t-2}$ \\
$\Gamma_3^\epsilon$ & $2^{2r-t-3}+\epsilon2^{r-2}-1$ & $2^{2r-t-3}$ \\
$\Gamma_4^\epsilon$ & $q^{2r-t-2}(q-1)+\epsilon q^{r-1}(q-2)-1$ & $q^{2r-t-2}(q-1)$ \\
\bottomrule
\specialrule{0em}{1.5pt}{1.5pt}

\end{tabular}
\end{table}

\end{corollary}

\subsection{Construction II}\

In this section, we choose appropriate subgroups of orthogonal groups and unitary groups to act on the vertex sets of strongly regular graphs. We show that the orbits and unions of certain orbits can be intriguing sets.

In $\Gamma_{1}^\epsilon$ and $\Gamma_{4}^\epsilon$, let $V=\mathbb{F}_q^r\times\mathbb{F}_q^r\times\mathbb{F}_q$ and define the parabolic quadratic form $Q$ on $V$ as follows:
\[
Q((x,y,z))=xy^\mathrm{T}+z^2.
\]
The polar form $B$ of $Q$ is given by
\[
B((x,y,z),(u,v,w))=xv^\mathrm{T}+uy^\mathrm{T}+2zw.
\]

\begin{lemma}\label{5}
Let $K=\left \{
\left(
\begin{array}{ccc}
	I & a\cdot u^\mathrm{T}u+S & a\cdot2u^\mathrm{T}\\
	O & I & 0^\mathrm{T}\\
	0 & u & 1
\end{array}
\right): u\in\mathbb{F}_q^r, S^\mathrm{T}=-S\right \}$, where $I,S,O$ are square matrices of order $r$ over $\mathbb{F}_{q}$ and $a\in\mathbb{F}_{q}$ such that $4a+1=0$. Then $K$ is a subgroup of {\rm O}$(2r+1,q)$.
\end{lemma}
\begin{proof}
Denote
$\left(
\begin{array}{ccc}
	I & a\cdot u^\mathrm{T}u+S & a\cdot2u^\mathrm{T}\\
	O & I & 0^\mathrm{T}\\
	0 & u & 1
\end{array}
\right)$ by $A_{u,S}$. For any $A_{u_1,S_1}$, $A_{u_2,S_2}\in K$, we have \[A_{u_1,S_1}\cdot A_{u_2,S_2}^{-1}=A_{u_1-u_2,S_1-S_2+u_1^\mathrm{T}u_2-u_2^\mathrm{T}u_1}\in K.\]
For any $(x,y,z)\in V$ and $A_{u,S}\in K$, we have
\begin{align*}
Q((x,y,z)^{A_{u,S}})&=Q((x,x(au^\mathrm{T}u+S)+y+zu,2axu^\mathrm{T}+z))\\	&=x(x(au^\mathrm{T}u+S)+y+zu)^\mathrm{T}+(2axu^\mathrm{T}+z)^2\\
&=Q((x,y,z))+(4a+1)axu^\mathrm{T}xu^\mathrm{T}+(4a+1)zxu^\mathrm{T}.
\end{align*}
Since $4a+1=0$, then $Q((x,y,z)^{A_{u,S}})=Q((x,y,z))$. Thus $K$ is a subgroup of {\rm O}$(2r+1,q)$.
\end{proof}
For $A_{u,S}\in K,\langle x\rangle\in \PG(2r,q)$, we define $\langle x\rangle^{A_{u,S}}=\langle x^{A_{u,S}}\rangle$. The action is well-defined.
\begin{lemma}\label{orbit}
The group $K$ has $\frac{q^r+\epsilon}{2}$ orbits on the vertex set $X$ of $\Gamma_{1}^\epsilon$ and $\Gamma_{4}^\epsilon$, and each $K$-orbit has size $q^{r}$.
\end{lemma}
\begin{proof}
Take $\epsilon=+1$ as an example. For any $\langle(x,y,z)\rangle\in X$, $Q((x,y,z))\in\square_q$.
If $x=0$, then
\[ \langle(0,y,z)\rangle^K=\{\langle(0,y+zu,z)\rangle:u\in\mathbb{F}_q^r\}=\langle(0,\mathbb{F}_q^r,1)\rangle.
\]
If $x\neq0$, let $T_{\lambda x}=\{u\in\mathbb{F}_q^r:xu^\mathrm{T}=\lambda\}$ for $\lambda\in \mathbb{F}_{q}$. Note that $T_{ix}+T_{jx}=T_{(i+j)x}$, $T_{0x}=\{u\in\mathbb{F}_q^r:xu^\mathrm{T}=0\}=\{xS:S^\mathrm{T}=-S\}$.
We have
\begin{equation}\label{eq1}
\begin{split}
&\langle(x,y,0)\rangle^K\\
=&\{\langle(x,x(au^\mathrm{T}u+S)+y,2axu^\mathrm{T})\rangle:u\in\mathbb{F}_q^r,S^\mathrm{T}=-S\}\\		=&\{\langle(x,axu^\mathrm{T}u+T_{0x}+y,2axu^\mathrm{T})\rangle:u\in\mathbb{F}_q^r\}\\
=&\{\langle(x,T_{0x}+y,0)\rangle\}\cup\{\langle(x,a\lambda^2u+T_{0x}+y,2a\lambda)\rangle:u\in T_{1x},\lambda\in\mathbb{F}_q^*\}.
\end{split}
\end{equation}
Thus $|\langle(x,y,0)\rangle^K|=q^{r-1}+q^{r-1}(q-1)=q^{r}=|\langle(0,y,z)\rangle^K|$. Since $Q((x,y,0))\in\square_q$, we have $y\in T_{\alpha x}$ with $\alpha\in\square_{q}$. Note that for $y_{1}\in T_{\tilde{\alpha} x}$, $\langle(x,y,0)\rangle^K=\langle(x,y_{1},0)\rangle^K$ if and only if $\alpha=\tilde{\alpha}$. Since $|X|=\frac{q^{r}(q^{r}+1)}{2}$, there are exactly $\frac{q^r+1}{2}$ orbits: $\langle(0,0,1)\rangle^K$, $\langle(x,y,0)\rangle^K$ with $y\in T_{\alpha x}$, where $x$ runs through $\mathbb{F}_q^r\setminus\{0\}$ and $\alpha\in\square_{q}$.

Similarly, in the case $\epsilon=-1$, there are exactly $\frac{q^r-1}{2}$ orbits: $\langle(x,y,0)\rangle^K$ with $y\in T_{\beta x}$, where $x$ runs through $\mathbb{F}_q^r\setminus\{0\}$ and $\beta\in\blacksquare_{q}$.
\end{proof}

\begin{thm}
Each $K$-orbit $O$ is an intriguing set of $\Gamma_{1}^\epsilon$ with
\[
|N(P)\cap O |
=\left \{
\begin{array}{ll}
	(1-\epsilon)q^{r-1}           &\text{if}\;P\in O,\\
	q^{r-1}                       &\text{if}\;P\in X\setminus O.\\
\end{array}
\right.
\]
\end{thm}
\begin{proof}

Take $q=5$ and $a=1$ as an example. In the case $\epsilon=+1$, $\frac{5^r+1}{2}$ $K$-orbits are as follows: $\langle(0,0,1)\rangle^K$, $\langle(x,y,0)\rangle^K$ with $y\in T_{1x}$, $\langle(x,\tilde{y},0)\rangle^K$ with $\tilde{y}\in T_{-1x}$, where $x$ runs through $\mathbb{F}_5^r\setminus\{0\}$. By $(\ref{eq1})$,
we have \[\langle(x,y,0)\rangle^K=\{\langle(x,T_{1x},0)\rangle\}\cup\{\langle(x,T_{0x},\lambda)\rangle:\lambda\in\square_{5}\}\cup\{\langle(x,T_{2x},\lambda)\rangle:\lambda\in\blacksquare_{5}\},\]
\[\langle(x,\tilde{y},0)\rangle^K=\{\langle(x,T_{-1x},0)\rangle\}\cup\{\langle(x,T_{0x},\lambda)\rangle:\lambda\in\blacksquare_{5}\}\cup\{\langle(x,T_{-2x},\lambda)\rangle:\lambda\in\square_{5}\}.\]
Next we show that each $K$-orbit $O$ is an intriguing set.

Note that if $P_{1},P_{2}$ are in the same $K$-orbit, then $|N(P_{1})\cap O |=|N(P_{2})\cap O |$ by Lemma \ref{5}. We first consider $K$-orbit $\langle(0,0,1)\rangle^K$. Take $\langle(0,0,1)\rangle\in \langle(0,0,1)\rangle^K$, then
\[
|N(\langle(0,0,1)\rangle)\cap \langle(0,0,1)\rangle^K|=0.
\]
Take any $\langle(x,y,0)\rangle\in X\setminus \langle(0,0,1)\rangle^K$ with $x\neq0$ and $y\in  T_{\alpha x}$, where $\alpha\in\square_{5}$, then
\[
|N(\langle(x,y,0)\rangle)\cap \langle(0,0,1)\rangle^K|=5^{r-1}.
\]
Next we think over the $K$-orbit $\langle(x,y,0)\rangle^K$ with $x\neq0$ and $y\in T_{1x}$ ($y\in T_{-1x}$ is similar). Take $\langle(x,y,0)\rangle\in \langle(x,y,0)\rangle^K$, one can prove that
\[
N(\langle(x,y,0)\rangle)\cap \langle(x,y,0)\rangle^K|=0.
\]
For any $P\in X\setminus \langle(x,y,0)\rangle^K$, if $P\in \langle(0,0,1)\rangle^K$, we suppose $P=\langle(0,0,1)\rangle$. Then
\[
|N(P)\cap \langle(x,y,0)\rangle^K|=5^{r-1}.
\]
If $P\in\langle(x_{1},v,0)\rangle^K$ with $x_{1}\neq x$ and $v\in T_{1x_{1}}$, we suppose $P=\langle(x_{1},v,0)\rangle$. Then
\[
|N(P)\cap \langle(x,y,0)\rangle^K|=5^{r-2}+2\cdot5^{r-2}+2\cdot5^{r-2}=5^{r-1}.
\]
If $P\in\langle(x_{2},\tilde{v},0)\rangle^K$ with $\tilde{v}\in T_{-1x_{2}}$, we can compute $|N(P)\cap \langle(x,y,0)\rangle^K|=5^{r-1}$ similarly.

In the case $\epsilon=-1$, $\frac{5^r-1}{2}$ $K$-orbits are as follows: $\langle(x,y,0)\rangle^K$ with $y\in T_{2x}$, $\langle(x,\tilde{y},0)\rangle^K$ with $\tilde{y}\in T_{-2x}$, where $x$ runs through $\mathbb{F}_5^r\setminus\{0\}$. By $(\ref{eq1})$, we have \[\langle(x,y,0)\rangle^K=\{\langle(x,T_{2x},0)\rangle\}\cup\{\langle(x,T_{1x},\lambda)\rangle:\lambda\in\square_{5}\}\cup\{\langle(x,T_{-2x},\lambda)\rangle:\lambda\in\blacksquare_{5}\},\]
\[\langle(x,\tilde{y},0)\rangle^K=\{\langle(x,T_{-2x},0)\rangle\}\cup\{\langle(x,T_{2x},\lambda)\rangle:\lambda\in\blacksquare_{5}\}\cup\{\langle(x,T_{-1x},\lambda)\rangle:\lambda\in\square_{5}\}.\]
We check that each $K$-orbit is an intriguing set by computing parameters omitted here.
\end{proof}

\begin{thm}\label{intriguing4}
The $K$-orbit $\langle(0,0,1)\rangle^K$ and the unions of $K$-orbits $\bigcup\limits_{\alpha\in\square_{q}}\langle(x,y,0)\rangle^K$ with $y\in T_{\alpha x}$, where $x$ runs through $\mathbb{F}_{q}^{r}\setminus\{0\}$ are positive intriguing sets of $\Gamma_{4}^+$ and
\[
|N(P)\cap \langle(0,0,1)\rangle^K |
=\left \{
\begin{array}{ll}
  q^{r}-1           &\text{if}\;P\in \langle(0,0,1)\rangle^K,\\
  2q^{r-1}     &\text{if}\;P\in X\setminus \langle(0,0,1)\rangle^K,
\end{array}
\right.
\]
\[
\left|N(P)\cap\left(\bigcup\limits_{\alpha\in\square_{q}}\langle(x,y,0)\rangle^K\right)\right|
=\left \{
\begin{array}{ll}
  (2q-3)q^{r-1}-1           &\text{if}\; P\in\bigcup\limits_{\alpha\in\square_{q}}\langle(x,y,0)\rangle^K,\\
  (q-1)q^{r-1}     &\text{if}\;P\in X\setminus\bigcup\limits_{\alpha\in\square_{q}}\langle(x,y,0)\rangle^K.
\end{array}
\right.
\]
\end{thm}
\begin{proof}
We first consider the $K$-orbit $\langle(0,0,1)\rangle^K$. For any $P=\langle(0,v,1)\rangle\in \langle(0,0,1)\rangle^K$, we have \[|N(P)\cap \langle(0,0,1)\rangle^K|=|\{\langle(0,u,1)\rangle:u\in\mathbb{F}_{q}^{r}\setminus\{v\}\}|=q^{r}-1.\]
For any $P=\langle(x,y,0)\rangle\in X\setminus \langle(0,0,1)\rangle^K$ with $y\in T_{\alpha x}$ and $\alpha\in\square_{q}$, we have \[|N(P)\cap \langle(0,0,1)\rangle^K|=|\{\langle(0,v,1)\rangle:v\in\mathbb{F}_{q}^{r},(xv^\mathrm{T})^{2}=4\alpha\}|=2q^{r-1}.\]

Next we show that for any $x\in\mathbb{F}_q^r\setminus\{0\}$, $\bigcup\limits_{\alpha\in\square_{q}}\langle(x,y,0)\rangle^K$ with $y\in T_{\alpha x}$ is an intriguing set.
Take any $P\in\bigcup\limits_{\alpha\in\square_{q}}\langle (x,y,0)\rangle^K$, then there exists $\alpha_{1}\in\square_{q}$ such that $P\in \langle (x,y_{1},0)\rangle^K$ with $y_{1}\in T_{\alpha_{1}x}$. We may suppose $P=\langle(x,y_{1},0)\rangle$. Note that \[\left|N(P)\cap\left(\bigcup\limits_{\alpha\in\square_{q}}\langle(x,y,0)\rangle^K\right)\right|=\sum\limits_{\alpha\in \square_{q}}|N(P)\cap \langle(x,y,0)\rangle^K|,\]
\begin{align*}
|N(P)\cap \langle(x,y,0)\rangle^K|=&|N(P)\cap\{\langle(x,T_{0x}+y,0)\rangle\}|\\+&|N(P)\cap\{\langle(x,a\lambda^2u+T_{0x}+y,2a\lambda)\rangle:u\in T_{1x},\lambda\in\mathbb{F}_q^*\}|.
\end{align*}
In fact, \[|N(P)\cap\{\langle(x,T_{0x}+y,0)\rangle\}|=|\{(x,m+y,0):(m+y)\in T_{\alpha x},(x(m+y)^\mathrm{T}+xy_{1}^\mathrm{T})^{2}=4\alpha_{1}\alpha\}|.\] From $(x(m+y)^\mathrm{T}+xy_{1}^\mathrm{T})^{2}=4\alpha_{1}\alpha$ and $(m+y)\in T_{\alpha x}$, we know only $\alpha=\alpha_{1}$ can satisfy this equation. Therefore \[\sum\limits_{\alpha\in \square_{q}}|N(P)\cap\{\langle(x,T_{0x}+y,0)\rangle\}|=|T_{\alpha_{1}x}|=q^{r-1}-1.\] Note that
\begin{align*}
&|N(P)\cap\{\langle(x,a\lambda^2u+T_{0x}+y,2a\lambda)\rangle:u\in T_{1x},\lambda\in\mathbb{F}_q^*\}|\\
=&|\{\langle(x,a\lambda^{2}u+n+y,2a\lambda)\rangle:\lambda\in\mathbb{F}_q^*,a\lambda^{2}u+n+y\in T_{(a\lambda^{2}+\alpha)x},\\&(x(a\lambda^{2}u+n+y)^\mathrm{T}+xy_{1}^\mathrm{T})^{2}=4\alpha_{1}\alpha\}|. \end{align*}
If $\alpha=\alpha_{1}$, from $(x(a\lambda^{2}u+n+y)^\mathrm{T}+xy_{1}^\mathrm{T})^{2}=4\alpha_{1}\alpha$, we have $\lambda=\pm a^{-1}\sqrt{\alpha_{1}}$. If $\alpha\neq\alpha_{1}$, $\lambda$ has four different solutions. Therefore,
\begin{align*}
&\sum\limits_{\alpha\in \square_{q}}|N(P)\cap\{(x,a\lambda^2u+T_{0x}+y,2a\lambda):u\in T_{1x},\lambda\in\mathbb{F}_q^*\}|\\
=&2|T_{(a\lambda^{2}+\alpha_{1})x}|+4\sum\limits_{\alpha\in\square_{q},\alpha\neq\alpha_{1}}|T_{(a\lambda^{2}+\alpha)x}|\\
=&2q^{r-1}+(\frac{q-1}{2}-1)\cdot4q^{r-1}=(2q-4)q^{r-1}.
\end{align*}
Thus \[\left|N(P)\cap\left(\bigcup\limits_{\alpha\in\square_{q}}\langle(x,y,0)\rangle^K\right)\right|=q^{r-1}-1+(2q-4)q^{r-1}=(2q-3)q^{r-1}-1.\]

Take any $P_{1}\in X\setminus\bigcup\limits_{\alpha\in\square_{q}}\langle(x,y,0)\rangle^K$. If $P_{1}\in \langle(0,0,1)\rangle^K$, by Lemma \ref{5}, we may suppose $P_{1}=\langle(0,0,1)\rangle$. Then \[\left|N(P_{1})\cap\left(\bigcup\limits_{\alpha\in\square_{q}}\langle(x,y,0)\rangle^K\right)\right|=(q-1)q^{r-1}.\]
If $P_{1}\notin \langle(0,0,1)\rangle^K$, we may suppose $P_{1}=\langle(x_{1},y_{1},0)\rangle$, where $x_{1}\neq x$ and $y_{1}\in T_{\beta x_{1}}$ for some $\beta\in\square_{q}$. Then \[|N(P_{1})\cap\{\langle(x,T_{0x}+y,0)\rangle\}|=|\{\langle(x,m+y,0)\rangle:(m+y)\in T_{\alpha x},(x_{1}(m+y)^\mathrm{T}+xy_{1}^\mathrm{T})^{2}=4\alpha\beta\}|.\]  First we fix $\alpha$. For given $x_{1},y_{1}$ and $x$, the number of pairs of $(x,m+y,0)$ satisfying the conditions is the number of solutions for $z$ of the following equations in $\mathbb{F}_{q}^{r}$:
$$
xz^\mathrm{T}=\alpha;
x_{1}z^\mathrm{T}=\pm2\sqrt{\alpha\beta}-xy_{1}^\mathrm{T}.
$$
Therefore \[\sum\limits_{\alpha\in \square_{q}}|N(P_{1})\cap\{\langle(x,T_{0x}+y,0)\rangle\}|=\frac{q-1}{2}\cdot2q^{r-2}=(q-1)q^{r-2}.\] Note that
\begin{align*}
&|N(P_{1})\cap\{\langle(x,a\lambda^2u+T_{0x}+y,2a\lambda)\rangle:u\in T_{1x},\lambda\in\mathbb{F}_q^*\}|\\
=&|\{(x,a\lambda^{2}u+n+y,2a\lambda):\lambda\in\mathbb{F}_q^*,a\lambda^{2}u+n+y\in T_{(a\lambda^{2}+\alpha)x},\\
&(x_{1}(a\lambda^{2}u+n+y)^\mathrm{T}+xy_{1}^\mathrm{T})^{2}=4\alpha\beta\}|.
\end{align*}
  First we fix $\alpha$ and $\lambda$. For given $x_{1},y_{1},x$, the number of pairs of $(x,a\lambda^{2}u+n+y,2a\lambda)$ satisfying the conditions is the number of solutions for $z$ of the following equations in $\mathbb{F}_{q}^{r}$:
$$
xz^\mathrm{T}=a\lambda^{2}+\alpha;
x_{1}z^\mathrm{T}=\pm2\sqrt{\alpha\beta}-xy_{1}^\mathrm{T}.
$$
Therefore \[\sum\limits_{\alpha\in \square_{q}}|N(P_{1})\cap\{\langle(x,a\lambda^2u+T_{0x}+y,2a\lambda)\rangle:u\in T_{1x},\lambda\in\mathbb{F}_q^*\}|=\frac{(q-1)^{2}}{2}\cdot2q^{r-2}=(q-1)^{2}q^{r-2}.\] Thus \[\left|N(P_{1})\cap\left(\bigcup\limits_{\alpha\in\square_{q}}\langle(x,y,0)\rangle^K\right)\right|=(q-1)q^{r-2}+(q-1)^{2}q^{r-2}=(q-1)q^{r-1}.\]
\end{proof}

\begin{thm}
The unions of $K$-orbits$ \bigcup\limits_{\alpha\in\blacksquare_{q}}\langle(x,y,0)\rangle^K$ with $y\in T_{\alpha x}$, where $x$ runs through $\mathbb{F}_{q}^{r}\setminus\{0\}$ are negative intriguing sets of $\Gamma_{4}^-$ and
\[
\left|N(P)\cap\left(\bigcup\limits_{\alpha\in\blacksquare_{q}}\langle(x,y,0)\rangle^K\right)\right|
=\left \{
\begin{array}{ll}
   q^{r-1}-1          &\text{if}\; P\in\bigcup\limits_{\alpha\in\blacksquare_{q}}\langle(x,y,0)\rangle^K,\\
   (q-1)q^{r-1}       &\text{if}\;P\in X\setminus\bigcup\limits_{\alpha\in\blacksquare_{q}}\langle(x,y,0)\rangle^K.
\end{array}
\right.
\]
\end{thm}
\begin{proof}
It is similar to Theorem \ref{intriguing4}. Note that for $P\in\bigcup\limits_{\alpha\in\blacksquare_{q}}\langle (x,y,0)\rangle^K$,
\[|N(P)\cap\{\langle(x,a\lambda^2u+T_{0x}+y,2a\lambda)\rangle:u\in T_{1x},\lambda\in\mathbb{F}_q^*\}|=0.\]
\end{proof}

In $\Gamma_{2}^+$ and $\Gamma_{3}^+$, let $V=\mathbb{F}_q^r\times\mathbb{F}_q^r$ and define the hyperbolic quadratic form $Q$ on $V$ as follows:
\[
Q((x,y))=xy^\mathrm{T}.
\]
The polar form $B$ of $Q$ is given by
\[
B((x,y),(u,v))=xv^\mathrm{T}+uy^\mathrm{T}.
\]
\begin{thm}

Let $L=\left \{
\left(
\begin{array}{ccc}
	I & S\\
	O & I \\
\end{array}
\right): S^\mathrm{T}=-S\right \}$, where $I,S,O$ are square matrices of order $r$ over $\mathbb{F}_{q}$. The group $L$ has $\frac{q^r-1}{q-1}$ orbits on the vertex set $X$ of $\Gamma_{2}^+$ and $\Gamma_{3}^+$, and the $L$-orbits are $\langle(u,T_{1u})\rangle$, where $u$ runs through $\mathbb{F}_q^r\setminus\{0\}$. In $\Gamma_{2}^+$, each $L$-orbit is a negative intriguing set and
\[
|N(P)\cap \langle(u,T_{1u})\rangle |
=\left \{
\begin{array}{ll}
	0           &\text{if}\;P\in \langle(u,T_{1u})\rangle,\\
	3 ^{r-2}     &\text{if}\;P\in X\setminus \langle(u,T_{1u})\rangle.\\
\end{array}
\right.
\]
In $\Gamma_{3}^+$, each $L$-orbit is a positive intriguing set and
\[
|N(P)\cap \langle(u,T_{1u})\rangle |
=\left \{
\begin{array}{ll}
	2^{r-1}-1           &\text{if}\;P\in \langle(u,T_{1u})\rangle,\\
	2^{r-2}     &\text{if}\;P\in X\setminus \langle(u,T_{1u})\rangle.\\
\end{array}
\right.
\]
\end{thm}
\begin{proof}
It is straightforward to check that $L$ is a group. For any $(u,v)\in X$, $uv^\mathrm{T}=1$, we have $(u,v)^{L}=\{(u,uS+v)\}=(u,T_{1u})$, hence $L$ is a subgroup of {\rm O}$^{+}(2r,q)$.

For $q=3$, take $\langle(u,\widetilde{v})\rangle\in \langle(u,T_{1u})\rangle$, then we have \[B((u,\widetilde{v}),(u,v))=uv^\mathrm{T}+u\widetilde{v}^\mathrm{T}=2\neq 0\] for any $(u,v)\in (u,T_{1u})$. Therefore \[|N(\langle(u,\widetilde{v})\rangle)\cap\langle(u,T_{1u})\rangle|=0.\] Take $\langle(u_{1},v_{1})\rangle\in X\setminus\langle(u,T_{1u})\rangle$, where $v_{1}\in T_{1u_{1}}$, then we have \[|N(\langle(u_{1},v_{1})\rangle)\cap\langle(u,T_{1u})\rangle|=|\{v\in\mathbb{F}_{q}^{r}:u_{1}v^\mathrm{T}+uv_{1}^\mathrm{T}=0,uv^\mathrm{T}=1\}|=3^{r-2}.\]

The case $q=2$ is similar and we omit the details.
\end{proof}

Let $\Gamma_{5}$ be defined in Section \ref{graph5}. Suppose $\omega$ is a primitive element of $\mathbb{F}_{q^{2r}}$, $e=\frac{q^{2r}-1}{q^2-1}$. Then $\mathbb{F}_{q^2}^*=\langle\omega^e\rangle$ and $\mathbb{F}_q^*=\langle\omega^{e(q+1)}\rangle$.
\begin{lemma}
Let $G=\{\sigma_{c_0,c_1,\cdots,c_{r-1}}:c_0\in\mathbb{F}_{q^r},c_i\in\mathbb{F}_{q^{2r}},c_i+c_{r-i}^{q^{r+2i}}=0,i=1,\cdots,r-1\}$, where
\begin{align*}
\sigma_{c_0,c_1,\cdots,c_{r-1}}:&V\rightarrow V,\quad
(u,v)\mapsto(u,v+\sum_{i=0}^{r-1}c_iu^{q^{2i}}).
\end{align*}
Then $G$ is a subgroup of {\rm U}$(2r,q^{2})$.
\end{lemma}
\begin{proof}
We need to show that for any $(u_{1},v_{1}),(u_{2},v_{2})\in V$ and any $\sigma_{c_0,c_1,\cdots,c_{r-1}}\in G$ (we write $\sigma$ to instead $\sigma_{c_0,c_1,\cdots,c_{r-1}}$ if there is no risk to make confusion), \[H((u_{1},v_{1})^{\sigma},(u_{2},v_{2})^{\sigma})=H((u_{1},v_{1}),(u_{2},v_{2})).\] It suffices to show that \[\Tr_{q^{2r}/q^{2}}(u_{1}(\sum\limits_{i=0}^{r-1}c_{i}u_{2}^{q^{2i}})^{q^{r}}+(\sum\limits_{i=0}^{r-1}c_{i}u_{1}^{q^{2i}})u_{2}^{q^{r}})=0.\] It is equivalent to show that \[\sum\limits_{j=0}^{r-1}\sum\limits_{i=1}^{r-1}(u_{1}c_{i}^{q^{r}}u_{2}^{q^{r+2i}}+c_{i}u_{1}^{q^{2i}}u_{2}^{q^{r}})^{q^{2j}}=0.\] Since $c_i+c_{r-i}^{q^{r+2i}}=0$, we have \[u_{1}c_{i}^{q^{r}}u_{2}^{q^{r+2i}}=u_{1}c_{r-i}^{q^{2i}}u_{2}^{q^{r+2i}}=(u_{1}^{q^{2(r-i)}}c_{r-i}u_{2}^{q^{r}})^{q^{2i}}.\] The claim follows from $\sum\limits_{j=0}^{r-1}(c_{r-i}u_{1}^{q^{2(r-i)}}u_{2}^{q^{r}})^{q^{2j}}=\sum\limits_{j=0}^{r-1}(c_{r-i}u_{1}^{q^{2(r-i)}}u_{2}^{q^{r}})^{q^{2j+2i}}$.
\end{proof}
\begin{lemma}\label{7}
For any $u\in\mathbb{F}_{q^{2r}}$, put $A=\{\sum\limits_{i=0}^{r-1}c_iu^{q^{2i}}:c_0\in\mathbb{F}_{q^r},c_i\in\mathbb{F}_{q^{2r}},c_i+c_{r-i}^{q^{r+2i}}=0\}$ and $B=\{x\in\mathbb{F}_{q^{2r}}:\Tr_{q^{2r}/q}(u^{q^r}x)=0\}$, then $A=B$.
\end{lemma}
\begin{proof}
For any $\sum\limits_{i=0}^{r-1}c_iu^{q^{2i}}\in A$, we have \[\Tr_{q^{2r}/q}(u^{q^r}(\sum\limits_{i=0}^{r-1}c_iu^{q^{2i}}))=\sum\limits_{i=0}^{r-1}\Tr_{q^{2r}/q}(c_iu^{q^r+q^{2i}})=0\] since $h((u,v)^{\sigma})=h((u,v))$. Hence $\sum\limits_{i=0}^{r-1}c_iu^{q^{2i}}\in B$, $A\subseteq B$. Note that \[| A|=|\{c_0+\sum\limits_{i=1}^{\frac{r-1}{2}}(c_i+c_i^{q^{r-2i}}):c_0\in\mathbb{F}_{q^r},c_i\in\mathbb{F}_{q^{2r}}\}|\geq| \{c_1+c_1^{q^{r-2}}:c_1\in\mathbb{F}_{q^{2r}}\}|.\] We define a linear map $f$ from $\mathbb{F}_{q^{2r}}$ to $\mathbb{F}_{q^{2r}}$ with $f(x)=x+x^{q^{r-2}}$, then $\text{ker}(f)=\mathbb{F}_{q^{2r}}\cap\mathbb{F}_{q^{r-2}}=\mathbb{F}_q$ and $\text{Im}(f)\cong\mathbb{F}_{q^{2r}}/\text{ker}(f)=\mathbb{F}_{q^{2r}}/\mathbb{F}_q$. Thus $| A|\geq| \text{Im}(f)|=q^{2r-1}=|B|$. The claim follows.
\end{proof}
\begin{lemma}\label{13}
Let $\omega$ be a primitive element of $\mathbb{F}_{q^{2r}}$ and $e=\frac{q^{2r}-1}{q^2-1}$. For any $l,l_0\in\{0,1,\cdots,q-2\}$, $l\neq l_0$, $m\in\{0,1,\cdots,q\}$, we have $\omega^{e(q+1)l_0}+\omega^{e(q+1)l}-\Tr_{q^2/q}(\omega^{e(l+l_0)+em(q-1)})\neq0.$
\end{lemma}
\begin{proof}
Let $\gamma=\omega^e\in\mathbb{F}_{q^{2}}^{\ast}$ and $\tilde{l}=l+m(q-1)$ for any $m\in\{0,1,\cdots,q\}$. Then \[\omega^{e(q+1)l_0}+\omega^{e(q+1)l}-\Tr_{q^2/q}(\omega^{e(l+l_0)+em(q-1)})
=\gamma^{(q+1)l_0}+\gamma^{(q+1)\tilde{l}}-\Tr_{q^2/q}(\gamma^{l_0+\tilde{l}}),\] where $\tilde{l}\in\{m(q-1),m(q-1)+1,\cdots,m(q-1)+q-2\}$. Let $x=\gamma^{l_0}$ and $y=\gamma^{\tilde{l}}$. Then $x^{q+1}+y^{q+1}-\Tr_{q^2/q}(xy)=0$ if and only if $x^q=y$. However, $x^q=\gamma^{ql_0}\in\{1,\gamma^q,\cdots,\gamma^{q(q-2)}\}$ and $y=\gamma^{\tilde{l}}\in\{\gamma^{m(q-1)},\gamma^{m(q-1)+1},\cdots,\gamma^{m(q-1)+q-2}\}$. Since $l\neq l_0$, $x^q\neq y$. This completes the proof.
\end{proof}
\begin{thm}
$G$ has $\frac{q^{2r}-1}{q+1}$ orbits on the vertex set $X$ of $\Gamma_{5}$. The union of the $G$-orbits $M_{k}:=\bigcup\limits_{i\in\mathbb{F}_q^*}\langle(\omega^k,v_{ki})\rangle^G$ is a negative intriguing set of $\Gamma_{5}$, where $k\in \{0,1,\cdots,e-1\}$ and $v_{ki}$ satisfies $\Tr_{q^{2r}/q}(\omega^{kq^r}v_{ki})=i$. Furthermore,
\[
|N(P)\cap M_{k}|
=\left \{
\begin{array}{cl}
	q^{2(r-1)}-1               &\text{if}\;P\in M_{k},\\
	q^{2r-3}(q^{2}-1)       &\text{if}\;P\in X\setminus M_{k}.
\end{array}
\right.
\]	
\end{thm}
\begin{proof}
For a given $j\in\mathbb{F}_q^*$, we take $\langle(\omega^k,v)\rangle\in X$ such that $h((\omega^k,v))=\Tr_{q^{2r}/q}(\omega^{kq^r}v)=j\in\mathbb{F}_q^*$, where $k\in\{0,1,\cdots,e-1\}$. Let $T_{ki}=\{x\in\mathbb{F}_{q^{2r}}:\Tr_{q^{2r}/q}(\omega^{kq^r}x)=i\}$. Note that $T_{ki_1}+T_{ki_2}=T_{k(i_1+i_2)}$. We get \[\langle(\omega^k,v)\rangle^G=\{\langle(\omega^k,v+x)\rangle:\Tr_{q^{2r}/q}(\omega^{kq^r}x)=0\}=\{\langle(\omega^k,\tilde{v})\rangle:\tilde{v}\in T_{kj}\}\] by Lemma \ref{7}. The number of $G$-orbits is $e(q-1)=\frac{q^{2r}-1}{q+1}$ since the first components of  elements in one $G$-orbit are same, and the $G$-orbits are as follows:
\[\langle(\omega^k,v_{ki})\rangle^G, \quad k\in\{0,1,\cdots,e-1\},i\in\mathbb{F}_q^*, \] where $v_{ki}\in T_{ki}$.

Let
\[
M_k=\bigcup\limits_{i\in\mathbb{F}_q^*}\langle(\omega^k,v_{ki})\rangle^G.
\]	
Next we show that for any $k\in\{0,1,\cdots,e-1\}$, $M_k$ is a negative intriguing set of $\Gamma_{5}$.
Take $\langle(\omega^k,v_{ki_0})\rangle\in M_k$, where $v_{ki_0}\in T_{ki_0}$, then
\begin{equation}\label{14}
|N(\langle(\omega^k,v_{ki_0})\rangle)\cap M_k|=|N(\langle(\omega^k,v_{ki_0})\rangle)\cap \langle(\omega^k,v_{ki_0})\rangle^G|+\sum_{i\in\mathbb{F}_q^*\setminus i_0}| N(\langle(\omega^k,v_{ki_0})\rangle)\cap \langle(\omega^k,v_{ki})\rangle^G|.	
\end{equation}
We now calculate $|N(\langle(\omega^k,v_{ki_0})\rangle)\cap \langle(\omega^k,v_{ki_0})\rangle^G|$, using Lemma \ref{8}.
\begin{align*}
&|N(\langle(\omega^k,v_{ki_0})\rangle)\cap \langle(\omega^k,v_{ki_0})\rangle^G|\\
=&|\{\langle(\omega^k,v_{ki_0}+x)\rangle:x\in T_{k0}\setminus{\{0\}},\\&H((\omega^k,v_{ki_0}+x),(\omega^k,v_{ki_0}))^{q+1}=h((\omega^k,v_{ki_0}+x))h((\omega^k,v_{ki_0}))\}|\\
=&| \{\langle(\omega^k,v_{ki_0}+x)\rangle:x\in T_{k0}\setminus{\{0\}},\Tr_{q^{2r}/q^2}(\omega^kv_{ki_0}^{q^r}+v_{ki_0}\omega^{kq^r}+x\omega^{kq^r})^{q+1}=i_0^2\}|\\
=&| \{\langle(\omega^k,v_{ki_0}+x)\rangle:x\in T_{k0}\setminus{\{0\}},(h((\omega^k,v_{ki_0}))+\Tr_{q^{2r}/q^2}(x\omega^{kq^r}))^{q+1}=i_0^2\}|\\
=&| \{\langle(\omega^k,v_{ki_0}+x)\rangle:x\in T_{k0}\setminus{\{0\}},(i_0+\Tr_{q^{2r}/q^2}(x\omega^{kq^r}))^{q+1}=i_0^2\}|
\end{align*}
Note that $x\in\mathbb{F}_{q^{2r}}$ such that $\Tr_{q^{2r}/q}(x\omega^{kq^r})=0$ if and only if $\Tr_{q^{2r}/q^2}(x\omega^{kq^r})\in\mathbb{F}_q$. Therefore
\begin{equation}\label{10}
\begin{split}
&|N(\langle(\omega^k,v_{ki_0})\rangle)\cap \langle(\omega^k,v_{ki_0})\rangle^G|\\
=&| \{\langle(\omega^k,v_{ki_0}+x)\rangle:x\neq0,\Tr_{q^{2r}/q^2}(x\omega^{kq^r})\in\mathbb{F}_q,(i_0+\Tr_{q^{2r}/q^2}(x\omega^{kq^r}))^2=i_0^2\}|\\
=&| \{\langle(\omega^k,v_{ki_0}+x)\rangle:x\neq0,\Tr_{q^{2r}/q^2}(x\omega^{kq^r})=0\}|\\	=&q^{2(r-1)}-1.
\end{split}	
\end{equation}
Next we calculate $|N(\langle(\omega^k,v_{ki_0})\rangle)\cap \langle(\omega^k,v_{ki})\rangle^G|$ in a similar way. We can first get
\begin{align*}
&| N(\langle(\omega^k,v_{ki_0})\rangle)\cap \langle(\omega^k,v_{ki})\rangle^G|\\
=&| \{\langle(\omega^k,v_{ki}+x)\rangle:x\in T_{k0},\Tr_{q^{2r}/q^2}(\omega^kv_{ki_0}^{q^r}+v_{ki}\omega^{kq^r}+x\omega^{kq^r})^{q+1}=ii_0\}|.
\end{align*}
Since $i,i_0\in\mathbb{F}_q^*=\langle\omega^{e(q+1)}\rangle$, we write $i=\omega^{e(q+1)l}$, $i_0=\omega^{e(q+1)l_0}$. Then there are $q+1$ solutions of the equation $x^{q+1}=ii_0=\omega^{e(q+1)(l+l_0)}$ in $\mathbb{F}_{q^2}^*$, i.e. $\omega^{e(l+l_0)+em(q-1)}$ for $m\in\{0,1,\cdots,q\}$. Therefore,
\begin{align*}
&| N(\langle(\omega^k,v_{ki_0})\rangle)\cap \langle(\omega^k,v_{ki})\rangle^G|\\
=&\sum_{m=0}^q| \{\langle(\omega^k,v_{ki}+x)\rangle:x\in T_{k0},\Tr_{q^{2r}/q^2}(\omega^kv_{ki_0}^{q^r}+v_{ki}\omega^{kq^r}+x\omega^{kq^r})=\omega^{e(l+l_0)+em(q-1)}\}|\\
=&\sum_{m=0}^q\sum_{x\in\mathbb{F}_{q^{2r}}}\sum_{\lambda\in\mathbb{F}_q}\frac{1}{q}\chi_q(\lambda \Tr_{q^{2r}/q}(x\omega^{kq^r}))\\
&\sum_{\mu\in\mathbb{F}_{q^2}}\frac{1}{q^2}\chi_{q^2}(\mu(\Tr_{q^{2r}/q^2}(\omega^kv_{ki_0}^{q^r}+v_{ki}\omega^{kq^r}+x\omega^{kq^r})-\omega^{e(l+l_0)+em(q-1)}))\\
=&\frac{1}{q^3}\sum_{m=0}^q\sum_{\lambda\in\mathbb{F}_q}\sum_{\mu\in\mathbb{F}_{q^2}}\sum_{x\in\mathbb{F}_{q^{2r}}}\chi_{q^{2r}}((\lambda+\mu)\omega^{kq^r}x+\mu(\omega^kv_{ki_0}^{q^r}+v_{ki}\omega^{kq^r}-\omega^{e(l+l_0)+em(q-1)})),
\end{align*}
where $\chi_{q}$, $\chi_{q^{2}}$ and $\chi_{q^{2r}}$, respectively, denote the canonical additive character of $\mathbb{F}_{q}$, $\mathbb{F}_{q^{2}}$ and $\mathbb{F}_{q^{2r}}$.
Note that
\begin{align*}	
&\sum_{x\in\mathbb{F}_{q^{2r}}}\chi_{q^{2r}}((\lambda+\mu)\omega^{kq^r}x+\mu(\omega^kv_{ki_0}^{q^r}+v_{ki}\omega^{kq^r}-\omega^{e(l+l_0)+em(q-1)}))\\
&=\left \{
\begin{array}{cl}
\chi_{q^{2r}}(\mu(\omega^kv_{ki_0}^{q^r}+v_{ki}\omega^{kq^r}-\omega^{e(l+l_0)+em(q-1)}))q^{2r}   &if\;\lambda+\mu=0,\\
	0                             &if\;\lambda+\mu\neq0.
\end{array}
\right.
\end{align*}
Therefore
\begin{align*}	
&|N(\langle(\omega^k,v_{ki_0})\rangle)\cap \langle(\omega^k,v_{ki})\rangle^G|=q^{2r-3}\sum_{m=0}^q\sum_{\lambda\in\mathbb{F}_q}\chi_{q^{2r}}(\lambda(\omega^kv_{ki_0}^{q^r}+v_{ki}\omega^{kq^r}-\omega^{e(l+l_0)+em(q-1)})).
\end{align*}
Since
\begin{align*}
&\chi_{q^{2r}}(\lambda(\omega^kv_{ki_0}^{q^r}+v_{ki}\omega^{kq^r}-\omega^{e(l+l_0)+em(q-1)}))\\
=&\chi_q(\lambda(\Tr_{q^{2r}/q}(\omega^kv_{ki_0}^{q^r}+v_{ki}\omega^{kq^r}-\omega^{e(l+l_0)+em(q-1)})))\\
=&\chi_q(\lambda(\omega^{e(q+1)l_0}+\omega^{e(q+1)l}-\Tr_{q^2/q}(\omega^{e(l+l_0)+em(q-1)}))),
\end{align*}
by Lemma \ref{13}, we have $\sum\limits_{\lambda\in\mathbb{F}_q}\chi_{q^{2r}}(\lambda(\omega^kv_{ki_0}^{q^r}+v_{ki}\omega^{kq^r}-\omega^{e(l+l_0)+em(q-1)}))=0$. Therefore
\begin{equation}\label{16}
|N(\langle(\omega^k,v_{ki_0})\rangle)\cap \langle(\omega^k,v_{ki})\rangle^G|=0. 	
\end{equation}
Hence, by $(\ref{14})$, $(\ref{10})$, $(\ref{16})$, we get $|N(\langle(\omega^k,v_{ki_0})\rangle)\cap M_k|=q^{2(r-1)}-1$.

Take $\langle(\omega^t,v_{ti_0})\rangle\in X\setminus M_k$, where $t\ne k$,$v_{ti_0}\in T_{ti_0}$, then
\begin{equation}\label{19}
|N(\langle(\omega^t,v_{ti_0})\rangle)\cap M_k|=\sum_{i\in\mathbb{F}_q^*}| N(\langle(\omega^t,v_{ti_0})\rangle)\cap \langle(\omega^k,v_{ki})\rangle^G|.
\end{equation}
Similarly, we calculate $| N(\langle(\omega^t,v_{ti_0})\rangle)\cap \langle(\omega^k,v_{ki})\rangle^G|$. We obtain
\begin{align*}
&| N(\langle(\omega^t,v_{ti_0})\rangle)\cap \langle(\omega^k,v_{ki})\rangle^G|\\
=&| \{\langle(\omega^k,v_{ki}+x)\rangle:x\in T_{k0},\Tr_{q^{2r}/q^2}(\omega^kv_{ti_0}^{q^r}+v_{ki}\omega^{tq^r}+x\omega^{tq^r})^{q+1}=ii_0\}|\\
=&\frac{1}{q^3}\sum_{m=0}^q\sum_{\lambda\in\mathbb{F}_q}\sum_{\mu\in\mathbb{F}_{q^2}}\sum_{x\in\mathbb{F}_{q^{2r}}}\chi_{q^{2r}}((\lambda\omega^{kq^r}+\mu\omega^{tq^r})x+\mu(\omega^kv_{ki_0}^{q^r}+v_{ki}\omega^{tq^r}-\omega^{e(l+l_0)+em(q-1)})),
\end{align*}
and notice that
\begin{align*}	
&\sum_{x\in\mathbb{F}_{q^{2r}}}\chi_{q^{2r}}((\lambda\omega^{kq^r}+\mu\omega^{tq^r})x+\mu(\omega^kv_{ki_0}^{q^r}+v_{ki}\omega^{tq^r}-\omega^{e(l+l_0)+em(q-1)}))\\
&=
\left \{
\begin{array}{cl}
\chi_{q^{2r}}(\mu(\omega^kv_{ki_0}^{q^r}+v_{ki}\omega^{kq^r}-\omega^{e(l+l_0)+em(q-1)}))q^{2r}   &if\;\lambda\omega^{kq^r}+\mu\omega^{tq^r}=0,\\
	0                             &if\;\lambda\omega^{kq^r}+\mu\omega^{tq^r}\neq0.
\end{array}
\right.
\end{align*}
If $\lambda\omega^{kq^r}+\mu\omega^{tq^r}=0$, either $\lambda=\mu=0$ or $\omega^{(k-t)q^r}\in\mathbb{F}_{q^2}^*$. Note that $\omega^{(k-t)q^r}\in\mathbb{F}_{q^2}^*$ if and only if $q^{2r}-1\mid(k-t)q^r(q^2-1)$,  which is equivalent to $q^{2r}-1\mid(k-t)(q^2-1)$. However, $(k-t)(q^2-1)\leqslant(e-1)(q^2-1)=q^{2r}-q^2<q^{2r}-1$. Therefore $\omega^{(k-t)q^r}\notin\mathbb{F}_{q^2}^*$ for any $k,t\in\{0,1,\cdots,e-1\}$. It follows that
\begin{equation}\label{20}
| N(\langle(\omega^t,v_{ti_0})\rangle)\cap \langle(\omega^k,v_{ki})\rangle^G|=q^{2r-3}(q+1).
\end{equation}
Hence $|N(\langle(\omega^t,v_{ti_0})\rangle)\cap M_k|=q^{2r-3}(q^{2}-1)$ by (\ref{19}) and (\ref{20}). Therefore for any $k\in\{0,1,\cdots,e-1\}$,
\[
|N(P)\cap M_k |
=\left \{
\begin{array}{cl}
q^{2(r-1)}-1              &if\;P\in M_k,\\
q^{2r-3}(q^{2}-1)       &if\;P\in X\setminus M_k.
\end{array}
\right.
\]
\end{proof}

The intriguing sets constructed in this section are summarized in the Table \ref{t4}.\\
\begin{table}[htbp]
\centering
\caption{Intriguing sets by group actions}\label{t4}
\begin{tabular}{ccccc}

\specialrule{0em}{1.5pt}{1.5pt}
\toprule
$srg$ & \rm{intriguing set} & $h_{1}$ & $h_{2}$ & \rm{Type} \\

\midrule
$\Gamma_1^+$  & $K$-\rm{orbit} & $0$ & $q^{r-1}$ & \rm{negative}\\
$\Gamma_1^-$ & $K$-\rm{orbit} & $2q^{r-1}$ & $q^{r-1}$ &\rm{positive}\\
$\Gamma_2^+$ & $L$-\rm{orbit} & $0$ & $3^{r-2}$ & \rm{negative}\\
$\Gamma_3^+$ & $L$-\rm{orbit} & $2^{r-1}-1$ & $2^{r-2}$ & \rm{positive}\\
$\Gamma_4^+ $ & $\bigcup\limits_{\alpha\in\square_{q}}\langle(x,y,0)\rangle^K$  & $(2q-3)q^{r-1}-1$ & $(q-1)q^{r-1}$ & \rm{positive}\\
$\Gamma_4^-$  & $\bigcup\limits_{\alpha\in\blacksquare_{q}}\langle(x,y,0)\rangle^K$ & $q^{r-1}-1$ & $(q-1)q^{r-1} $ & \rm{negative}\\
$\Gamma_5$ & $M_{k}$ & $q^{2(r-2)}-1$ & $q^{2r-3}(q^2-1)$  & \rm{negative}\\
\bottomrule
\specialrule{0em}{1.5pt}{1.5pt}

\end{tabular}
\end{table}

\subsection{Construction III}\

\begin{thm}\label{18}
Let $X$ be the vertex set of a strongly regular graph in the table below. Given a nonsingular point $\langle y\rangle$,  let $M=\{\langle x\rangle\in PG(2r,q):B(x,y)^{2}-Q(x)Q(y)\in\square_{q}\}$. Then $M\cap X$ is an intriguing set (see Table \ref{t5}). \\
\begin{table}[htbp]
\centering
\caption{Intriguing sets by nonsingular points}\label{t5}
\begin{tabular}{ccccc}

\specialrule{0em}{1.5pt}{1.5pt}
\toprule
$srg$ & $y$ & $h_{1}$ & $h_{2}$ & \rm{Type} \\

\midrule
$\Gamma_1^{+\perp}$  & $Q(y)\in\blacksquare_{q}$ & $\frac{q-1}{2}\cdot\frac{q^{r-1}(q^{r-1}+1)}{2}$ & $\frac{q-1}{2}\cdot\frac{q^{r-1}(q^{r-1}-1)}{2}$ & \rm{positive}\\
$\Gamma_1^{-\perp}$  & $Q(y)\in\square_{q}$ & $\frac{q-1}{2}\cdot\frac{q^{r-1}(q^{r-1}-1)}{2}$ & $\frac{q-1}{2}\cdot\frac{q^{r-1}(q^{r-1}+1)}{2}$ &\rm{negative}\\
$\Gamma_2^+$ & $Q(y)\in\blacksquare_{3}$ & $\frac{3^{r-2}(3^{r-1}-1)}{2}$ & $\frac{3^{r-2}(3^{r-1}+1)}{2}$ & \rm{negative}\\
\bottomrule
\specialrule{0em}{1.5pt}{1.5pt}

\end{tabular}
\end{table}

\end{thm}

\begin{proof}
Take ${\Gamma}_{1}^{+\perp}$ as an example. In the case $q=3$, we can check $M\cap X=\langle y\rangle^{\perp}\cap X$ since $Q(x)=1,Q(y)=-1$. For any $P_{1}=\langle x_{1}\rangle,P_{2}=\langle x_{2}\rangle\in M\cap X$, we define a linear map $\sigma$ from $\langle x_{1},y\rangle$ to $\langle x_{2},y\rangle$ such that $x_{1}^{\sigma}=x_{2}$ and $y^{\sigma}=y$. Then $\sigma$ is an isometry from $\langle x_{1},y\rangle$ to $\langle x_{2},y\rangle$. By Lemma \ref{witt}, $\sigma$ can extend to an isometry of $V$, hence $|N(P_1)\cap\langle y\rangle^{\perp}\cap X|=|N(P_2)\cap\langle y\rangle^{\perp}\cap X|$. For any $P_{3}=\langle x_{3}\rangle,P_{4}=\langle x_{4}\rangle\in X\setminus(M\cap X)$, we define a linear map $\sigma'$ from $\langle x_{3},y\rangle$ to $\langle x_{4},y\rangle$ such that $x_{3}^{\sigma'}=x_{4}$ and $y^{\sigma'}=\lambda y$, where $\lambda$ is determined by $B(x_{3},y)=\lambda B(x_{4},y)$. Then $\sigma'$ is an isometry from $\langle x_{3},y\rangle$ to $\langle x_{4},y\rangle$. Similarly, we obtain $|N(P_3)\cap\langle y\rangle^{\perp}\cap X|=|N(P_4)\cap\langle y\rangle^{\perp}\cap X|$. Therefore, $M\cap X$ is an intriguing set.

In the case $q=5$, for any $P_{1}=\langle x_{1}\rangle,P_{2}=\langle x_{2}\rangle\in M\cap X$, we may suppose $Q(x_1)=Q(x_2)=1$. Since $|\{a^{2}:a^{2}-Q(y)\in\square_{5}\}|=1$, $B(x_{1},y)^{2}=B(x_{2},y)^{2}$. Similar to the case $q=3$, we have an isometry $\theta$ of $V$ such that $x_{1}^{\theta}=x_{2}$, $y^{\theta}=\pm y$ in accordance with $B(x_{1},y)=\pm B(x_{2},y)$ and $M^\theta=M$. Hence $|N(P_1)\cap M\cap X|=|N(P_2)\cap M\cap X|$. For any $P_{3}=\langle x_{3}\rangle,P_{4}=\langle x_{4}\rangle\in X\setminus(M\cap X)$, we can obtain $|N(P_3)\cap M\cap X|=|N(P_4)\cap M\cap X|$ similarly. Therefore, $M\cap X$ is an intriguing set. We omit the calculation of the intersection numbers of $M\cap X$.

For ${\Gamma}_{1}^{-\perp}$ and ${\Gamma}_{2}^{+}$, the proofs are similar.
\end{proof}

\begin{remark}
We observe that the size of $M\cap X$ is independent of the choice of $\langle y\rangle$ in Theorem \ref{18}. If intriguing sets of the same type have the same size, then they have the same intersection numbers (see \cite[Proposition 3.7]{ref1}).
\end{remark}

\section{Concluding remarks}
In this paper, we describe three methods to construct intriguing sets in the five classes of strongly regular graphs defined on nonisotropic points of finite classical polar spaces. It is an interesting open problem to construct more intriguing sets in these graphs.

\section*{Acknowledgement}
The authors would like to thank the reviewers and Prof. Tao Feng for their constructive comments and suggestions that improved the quality and presentation of this paper. This work was supported by National Natural Science Foundation of China under Grant No.11771392.


\begin{thebibliography}{10}

\bibitem{ref2}
J. Bamberg, M. Law, and T. Penttila,
\newblock Tight sets and m-ovoids of generalised quadrangles,
\newblock {\em Combinatorica}, 29(1):1--17, 2009.

\bibitem{ref3}
J. Bamberg, S. Kelly, M. Law, and T. Penttila,
\newblock Tight sets and m-ovoids of finite polar spaces,
\newblock {\em J. Combin. Theory Ser. A}, 29:1--17, 2007.

\bibitem{ref7}
J. Bamberg, F.D. Clerck, and N. Durante,
\newblock Intriguing sets in partial quadrangles,
\newblock {\em J. Combin. Des.}, 19(3):217--245, 2011.

\bibitem{ref12}
A.E. Brouwer and W.H. Haemers,
\newblock Spectra of Graphs,
\newblock {\em Springer}, New York, 2012.

\bibitem{ref22}
A.E. Brouwer and H. Van Maldeghem,
\newblock Strongly regular Graphs,
\newblock \href{https://homepages.cwi.nl/~aeb/math/srg/rk3/srgw.pdf}{https://homepages.cwi.nl/~aeb/math/srg/rk3/srgw.pdf}.

\bibitem{ref11}
S. Ball,
\newblock Finite geometry and combinatorial applications,
\newblock {\em Cambridge University}, Cambridge, 2015.

\bibitem{ref15}
P.J. Cameron,
\newblock Notes on Classical Groups,
\newblock 2000.

\bibitem{ref10}
R. Calderbank and W.M. Kantor,
\newblock The geometry of two-weight codes,
\newblock {\em Bulletin of the London Mathematical Society}, (2):2, 1986.

\bibitem{ref17}
A. Cossidente and F. Pavese,
\newblock Intriguing sets of $W(5,q)$, q even,
\newblock {\em J. Combin. Theory Ser. A}, 127:303-313, 2014.

\bibitem{ref18}
A. Cossidente and F. Pavese,
\newblock Intriguing sets of quadrics in $PG(5,q)$,
\newblock {\em Adv. Geom.}, 17(3):339-345, 2017.

\bibitem{ref19}
A. Cossidente and F. Pavese,
\newblock On intriguing sets of finite symplectic spaces,
\newblock {\em Des. Codes Cryptogr.}, 86(5):1161-1174, 2018.

\bibitem{ref1}
B. De Bruyn and H. Suzuki,
\newblock Intriguing sets of vertices of regular graphs,
\newblock {\em Graphs Combin.}, 26(5):629--646, 2010.

\bibitem{ref20}
J. De Beule and K. Metsch,
\newblock On the smallest non-trivial tight sets in Hermitian polar spaces,
\newblock {\em Electron. J. Combin.}, 24(1):no. 1.62, 13pp, 2017.

\bibitem{ref5}
T. Feng and R. Tao,
\newblock An infinite family of $m$-ovoids of $Q(4,q)$,
\newblock {\em Finite Fields Appl.}, 63:101644, 2020.

\bibitem{ref6}
T. Feng, Y. Wang, and Q.~Xiang,
\newblock On $m$-ovoids of symplectic polar spaces,
\newblock {\em J. Combin. Theory Ser. A}, 175:105279, 2020.

\bibitem{ref4}
T. Feng, K. Momihara, and Qing Xiang,
\newblock A family of $m$-ovoids of parabolic quadrics,
\newblock {\em J. Combin. Theory Ser. A}, 140:97--111, 2016.

\bibitem{ref23}
X. Hubaut and R. Metz,
\newblock A class of strongly regular graphs related to orthogonal groups,
\newblock {\em Ann. Discrete Math.}, 18:469-472, 1983.

\bibitem{ref24}
J.W.P. Hirschfeld and J.A. Thas,
\newblock General Galois Geometries,
\newblock {\em Springer Monographs in Mathmatics}, London, 2016.

\bibitem{ref21}
K. Metsch,
\newblock Small tight sets in finite elliptic, parabolic and Hermitian polar spaces,
\newblock {\em Combinatorica}, 36(6):725-744, 2016.

\bibitem{ref8}
S.E. Payne,
\newblock Tight pointsets in finite generalized quadrangles,
\newblock {\em Eighteenth Southeastern International Conference on Combinatorics,
  Graph Theory, and Computing}, 60:243--260, 1987.

\bibitem{ref9}
J.A. Thas,
\newblock Interesting pointsets in generalized quadrangles and partial
  geometries,
\newblock {\em Linear Algebra Appl.}, 114-115(1):103--131, 1989.

\end{thebibliography}
\end{document}